\pdfoutput=1
\documentclass[12pt,a4paper]{article}
\usepackage[utf8]{inputenc}
\usepackage{amsmath}
\usepackage{amsfonts}
\usepackage{amssymb}
\usepackage{graphicx}
\usepackage{amsthm}
\usepackage{natbib}
\usepackage{caption}
\usepackage{hyperref}

\usepackage{xcolor}

\usepackage{enumerate}

\DeclareMathOperator{\trace}{trace}

\newtheorem{lemma}{Lemma}[section]
\newtheorem{corollary}{Corollary}[section]
\newtheorem{theorem}{Theorem}[section]

\newtheorem{remark}{Remark}

\newtheorem{assumption}{Assumption}

\numberwithin{equation}{section}

\usepackage{subcaption}

\author{Andreas Elsener and Sara van de Geer \\
Seminar for Statistics \\
ETH Z\"urich}
\title{Sparse spectral estimation with missing and corrupted measurements}
\begin{document}
\maketitle
\begin{abstract}
Supervised learning methods with missing data have been extensively studied not just due to the techniques related to low-rank matrix completion. Also in unsupervised learning one often relies on imputation methods. As a matter of fact, missing values induce a bias in various estimators such as the sample covariance matrix. In the present paper, a convex method for sparse subspace estimation is extended to the case of missing and corrupted measurements. This is done by correcting the bias instead of imputing the missing values. The estimator is then used as an initial value for a nonconvex procedure to improve the overall statistical performance. The methodological as well as theoretical frameworks are applied to a wide range of statistical problems. These include sparse Principal Component Analysis with different types of randomly missing data and the estimation of eigenvectors of low-rank matrices with missing values. Finally, the statistical performance is demonstrated on synthetic data.
\end{abstract}

\section{Introduction}
\subsection{Background and Motivation}
The vast majority of unsupervised learning methods assume that all variables are fully observed. This is clearly an unrealistic assumption when it comes to applications. It is more realistic to allow for missing observations. Depending on the nature of the problem one can think of different causes for missing data. For instance, in gene expression data missing data might arise as a consequence of device failures. Other types of corruptions might be due to frauds. Apart from possibly malicious manipulations one can also imagine removing observations for privacy reasons.

Following the book \citet{carroll2006measurement} we name some important applications where measurement errors arise. These include the measurement of blood pressure, the measurement of urinary sodium chloride level, and the exposure to pollutants. In \citet{carroll2006measurement} the main focus lies on how to appropriately adapt for the measurement errors in supervised learning.

In this paper, we propose a methodological as well as theoretical framework that can be applied to estimation problems where one is interested in computing (sparse) singular vectors (or eigenvectors).
In particular, a main focus lies on one of the most prominent tools from unsupervised learning: Principal Component Analysis (PCA) and its modifications. PCA goes back to \citet{Pearson1901} and \citet{hotelling1933analysis}. It is mainly used to reduce the dimensionality of a dataset. PCA has two main drawbacks. Firstly, it produces solutions that are linear combinations of all the variables in a given dataset, which is problematic when the number of variables is large. Secondly, it has been shown in \citet{johnstone2009consistency} to be inconsistent when the number of variables exceeds the number of observations. To overcome these limitations a plethora of methods based on the sparsity of the eigenvector corresponding to the largest eigenvalue of the population covariance matrix have been proposed and analyzed: \citet{ZouHastieTibshirani2006}, \citet{d2007direct}, \citet{amini2009}, \citet{vu2013minimax} and \citet{vu2013fantope}. These methods are usually grouped in what is called ``sparse PCA''. Given a data matrix $X \in \mathbb{R}^{n \times p}$ with i.i.d. rows and positive definite covariance matrix $\Sigma_0$ the aim is to compute the vector that maximizes the variance subject to some constraints:
\begin{align}
\label{eqn:sparsepca1}
\begin{split}
\text{maximize} \ &\widehat{\text{Var}}(X \beta) \\
\text{subject to} \ &\Vert \beta \Vert_2 = 1, \\
						   & \Vert \beta \Vert_1 \leq t,
\end{split}
\end{align}
with respect to $\beta \in \mathbb{R}^p$, where $t > 0$ is a tuning parameter. It has to be noticed that the estimation problem (\ref{eqn:sparsepca1}) is to be seen as a representative of the various approaches to sparse PCA. For simplicity, we assume throughout the paper, if not mentioned otherwise, that $X$ has mean zero, so that $\Sigma_0 := \mathbb{E} X^T X/n$. As a consequence, a ``sensible'' estimator for the variance is given by
\begin{align*}
\widehat{\text{Var}} (X \beta) = \beta^T \hat{\Sigma} \beta = \beta^T \frac{X^TX}{n} \beta = \frac{\Vert X \beta \Vert_2^2}{n}.
\end{align*}
Here, we consider the case where the matrix $X$ is corrupted by additional sources of random noise.   This includes the following cases:
\begin{enumerate}[i)]
\item Sparse PCA with missing data (uniformly at random).
\item Sparse PCA with random multiplicative noise.
\item Sparse PCA with missing data (non-uniformly at random).
\item Sparse eigenvectors of low-rank matrices with randomly missing data.
\end{enumerate}
In general, we have under missing data that $\mathbb{E} X^T X/n \neq \Sigma_0$, i.e. the missing data induce a bias in the sample covariance matrix.

Sparse PCA is also at the origin of the formalization of a phenomenon that is observed in many different problems in statistics: the trade-off between computability from an optimization point of view and the purely statistical performance. In a nutshell, this means that in order to effectively compute the solution of the optimization problem one has to pay a statistical price. The kick-off of this booming area is represented by the works \citet{berthet2013complexity} and \citet{berthet2013optimal}. Missing data or different types of corruptions clearly represent an additional challenge.

Our first main contribution is an application of a computationally feasible method based on a convex relaxation of sparse PCA to the estimation of sparse eigenvectors of matrices with missing data. This method was proposed by \citet{d2007direct} and \citet{vu2013fantope}. We show that this estimator is consistent also for the missing data case. The price to pay for a computationally feasible solution is a slower statistical rate of convergence, or equivalently a stronger requirement on the sparsity. Consider the example of randomly missing data with $0 < \delta \leq 1$ being the probability that we observe a single element of the matrix $X$. We obtain the following asymptotic rate for the estimator $\hat{\beta}_{\text{init}}$ of the loading vector of the first principal component $\beta^0$ with $s_0$ non-zero entries:
\begin{align}
\label{eqn:nonoptrate}
\Vert \hat{\beta}_{\text{init}} - \beta^0 \Vert_2 = \mathcal{O}_{\mathbb{P}} \left(s_0 \sqrt{\frac{\log p}{\delta^2 n}}\right).
\end{align}
The rate in equation (\ref{eqn:nonoptrate}) is not optimal as we have a scaling of the type $s_0^2/n$ instead of $s_0/n$. To overcome this limitation, we propose to use a nonconvex acceleration in a second step. A common feature to nonconvex estimation problems is that they typically require a ``good'' initial value to work. The ``good'' initialization is given by the convex relaxation. As a matter of fact, nonconvex optimization problems are often solved iteratively with no guarantee to end up at the global minimum. They rather output a stationary point $\tilde{\beta}$. For this reason, we derive statistical properties of any stationary point. As a consequence, we obtain a rate of the type
\begin{align*}
\Vert \tilde{\beta} - \beta^0 \Vert_2 = \mathcal{O}_{\mathbb{P}} \left( \sqrt{\frac{s_0 \log p}{\delta^2 n}} \right).
\end{align*}
We also examine cases where the distribution of the random missing data mechanism has different parameters depending on the row of $X$ as well as other types of multiplicative noise.
\subsection{Related literature}
Thanks to the proximity to and relevance for applications there is a very rich literature on how to properly deal with missing data in ``supervised learning''. The most famous example is matrix completion. The noiseless matrix completion literature comprises among many others the works of \citet{candes2009exact} and \citet{candes2010power}. The noisy case has been studied among many others in \citet{keshavan2010matrix}, \citet{negahban2011estimation} and \citet{Koltchinskii2011}. In addition, the work of \citet{loh2011high} considers a bias correction of the sample covariance matrix that leads to a nonconvex optimization problem for high-dimensional linear regression.

Another similar question is the estimation of eigenvectors of a low-rank matrix that is corrupted by additive noise. This problem has been studied in \citet{benaych2012singular}, \citet{shabalin2013reconstruction}, \citet{donoho2014minimax} and \citet{cai2018rate}. In contrast to our work, the previously mentioned methods are not related to the sparsity in the singular vectors.

As far as the unsupervised methods with missing data are concerned it has to be said that many of these rely on covariance matrices. It is therefore a main part of these estimation procedures to first properly estimate the covariance matrix in a missing data framework. This has been done in \citet{lounici2014high} and \citet{cai2016minimax}. In particular, PCA with deterministic missing data has recently been studied by \citet{zhang2018heteroskedastic}. An approach based on a bias correction to estimate the covariance matrix is employed in an iterative procedure that involves computing the SVD.

The work \citet{florescu2016spectral} attempts to correct for the bias in the diagonal entries to recover the partition in bipartite stochastic block models. The diagonal entries of the sample covariance matrix are always set to zero. In contrast, we derive a theoretical justification for the bias corrections proposed in our methods.
\subsection{Organization of the paper}
In Section \ref{s:methodology} the convex method as well as the nonconvex methods are presented. The theoretical guarantees for both techniques are discussed from a purely deterministic point of view in Section \ref{s:deterministicresults}. There, it is assumed that the tuning parameter properly bounds the noise part of the problem. In Section \ref{s:applications} the applications to sparse PCA with different random missing data mechanisms as well as the estimation of sparse eigenvectors of low-rank matrices are analyzed. In Section \ref{s:empirical} some simulation results are presented that confirm the theoretical findings.

\section{Methodology}
\label{s:methodology}
We denote an unbiased estimator of the quantity of interest (e.g. the covariance matrix) by $\tilde{\Sigma}$. Inspired by the estimator proposed in \citet{vu2013fantope} and \citet{d2007direct} we solve the following estimation problem with respect to $F \in \mathbb{R}^{p \times p}$:
\begin{align}
\label{eqn:fantopeest}
\begin{split}
&\text{minimize} \ - \trace(\tilde{\Sigma} F) + \mu \Vert F \Vert_1 \\
&\text{subject to} \ \trace(F) = 1, \\
&\phantom{aaaaaaaaa} 0 \preceq F \preceq I,
\end{split}
\end{align}
where $\mu > 0$ is a tuning parameter that needs to appropriately chosen. The optimization problem (\ref{eqn:fantopeest}) is an instance of a semidefinite program (SDP). We will refer to the solution of (\ref{eqn:fantopeest}) as the ``SDP estimator''. As a second step we speed up the statistical rate of convergence by solving a nonconvex optimization problem of the form
\begin{align}
\label{eqn:accelest}
\hat{\beta} = \underset{\beta \in \mathcal{C}}{\arg \min} \ \frac{1}{4} \left\Vert \tilde{\Sigma} - \beta \beta^T \right\Vert_F^2 + \lambda \Vert \beta \Vert_1,
\end{align}
where $\lambda > 0 $ is a tuning parameter and the neighborhood $\mathcal{B}$ is defined as $\mathcal{B} := \left\lbrace \beta \in \mathbb{R}^p : \Vert \beta - \beta^0 \Vert_2 \leq \eta \right\rbrace$. Moreover, for a tuning parameter $Q > 0$ the set of constraints is defined as $\mathcal{C} := \mathcal{B} \cap \left\lbrace \beta \in \mathbb{R}^p : \Vert \beta \Vert_1 \leq Q \right\rbrace$. A common feature of nonconvex optimization problems is that in order to succeed they typically need a ``good'' initial point. Let $\hat{u}_1$ be the eigenvector corresponding to the largest eigenvalue of $\hat{F}$. The vector $\hat{u}_1$ is then normalized. As an initial point for the problem (\ref{eqn:accelest}) we choose the appropriately rescaled solution of the convex problem (\ref{eqn:fantopeest}). We choose $\hat{\beta}_{\text{init}}$ as 
\begin{align}
\label{eqn:initialest}
\hat{\beta}_{\text{init}} := \vert \text{trace}(\tilde{\Sigma} \hat{Z}) \vert^{1/2} \hat{u}_1
\end{align}
We point out that in contrast to \citet{jankova2018biased} we need to take the absolute value of $ \text{trace}(\tilde{\Sigma} \hat{Z}) $ as this quantity might be negative due to $\tilde{\Sigma}$, which could be negative definite. As far as the nonconvex estimator (\ref{eqn:accelest}) is concerned, it is convenient to view it as a penalized empirical risk minimizer. The empirical risk and its theoretical counterpart, the risk, are defined for all $\beta$ as
\begin{align*}
R_n(\beta) = \frac{1}{4} \left\Vert \tilde{\Sigma} - \beta \beta^T \right\Vert_F^2, \quad R(\beta) = \mathbb{E} R_n(\beta).
\end{align*}
Their derivatives are respectively given by
\begin{align*}
\dot{R}_n(\beta) = \frac{\partial}{\partial \beta} R_n(\beta), \quad \dot{R} (\beta) = \mathbb{E} \dot{R}_n(\beta).
\end{align*}
The estimator $\hat{\beta}$ was proposed in \citet{geer2015} and further analyzed in \citet{jankova2018biased} and \citet{elsener2018sharp} in the special case of sparse PCA.
The following lemma parallels Lemma 2 of \citet{jankova2018biased}. We need to adapt it to the initial estimator (\ref{eqn:initialest}) compared to the one in \citet{jankova2018biased} as otherwise we would incur in taking roots of negative numbers. However, we find that Lemma 2 of \citet{jankova2018biased} can be applied to the new initial estimator.

\section{Deterministic results}
\label{s:deterministicresults}
The first theoretical guarantee that needs to be established is about the statistical performance of the initial estimator. We make use of the (deterministic) theory developed in \citet{vu2013fantope}.
\subsection{Initial estimator}
We denote the solution of the optimization problem (\ref{eqn:fantopeest}) by $\hat{F}$. The following lemma provides the theoretical guarantees for the estimator for the eigenvector corresponding to the largest eigenvalue of $\hat{F}$ as well as for its rescaled version (\ref{eqn:initialest}).
\begin{lemma}[Adapted from Lemma 2 in \citet{jankova2018biased}]
\label{lemma:initalest}
Let $\hat{F}$ be the solution of the optimization problem (\ref{eqn:fantopeest}). Assume that $\mu \geq 2 \Vert \tilde{\Sigma} - \Sigma_0 \Vert_\infty$. Assume without loss of generality that $\hat{u}_1^T u_1 \geq 0$. Then
\begin{align*}
\left\Vert \hat{u}_1 - u_1 \right\Vert_2^2 \leq \frac{4C s^2 \mu^2}{(\Lambda_1 - \Lambda_2)^2} =: \varepsilon^2,
\end{align*}
where $\Lambda_1$ and $\Lambda_2$ are the largest and second largest eigenvalues of $\Sigma_0$, respectively. Further, assuming that the largest eigenvalue of $\Sigma_0$ satisfies $\Lambda_{\max}(\Sigma_0) > \alpha$ we have that
\begin{align*}
\left\Vert \hat{\beta}_{\text{\emph{init}}} - \beta^0 \right\Vert_2 \leq \frac{1}{2 \sqrt{\Lambda_{\max}(\Sigma_0) - \alpha}} \alpha + 2 \sqrt{\varepsilon} \Lambda_{\max}(\Sigma_0)^{1/2},
\end{align*}
where
\begin{align*}
\alpha = s \mu + \varepsilon^2 + 6 \Vert \beta^0 \Vert_2^2 \varepsilon + 4 \Vert \beta^0 \Vert_2 \sqrt{\varepsilon}.
\end{align*}
\end{lemma}
The proof of Lemma \ref{lemma:initalest} can be found in the appendix.

\subsection{Nonconvex acceleration}
The identifiability on the set $\mathcal{C}$ is guaranteed by the following lemma. The risk is shown to be strongly convex on the set $\mathcal{C}$. We first state an assumption on the radius of the neighborhood $\mathcal{B}$ which depends on the signal strength (i.e. the magnitude of the largest singular value of $\Sigma_0$).
\begin{assumption}~
\label{assumption:eigenval}
Let $\phi_{\max} = \phi_1 \geq \dots \phi_p \geq 0$ be the singular values of $\Sigma_0$. Assume that for $\rho >0$
\begin{align*}
\phi_{\max} \geq \phi_j + \rho, \ \text{for all} \ j \neq 1.
\end{align*}
 We assume further that $\rho > 3 \eta$.
\end{assumption}
Assumption \ref{assumption:eigenval} guarantees a sufficiently high curvature in the neighborhood of $\beta^0$ by requiring a sufficiently large gap between the largest and second largest eigenvalues of the  covariance matrix $\Sigma_0$.
\begin{lemma}[Adapted from Lemma III.8. in \citet{elsener2018sharp} and Lemma 12.7 in \citet{geer2015}] Suppose that Assumption \ref{assumption:eigenval} is satisfied. We then have for all $\beta_1, \beta_2 \in \mathcal{C}$ that
\begin{align*}
&R(\beta_1) - R(\beta_2) - \dot{R}(\beta_2)^T(\beta_1 - \beta_2) \\
&\geq 2\phi_{\max}(\rho - 3 \eta) \Vert \beta_1 - \beta_2 \Vert_2^2 =: G(\Vert \beta_1 - \beta_2 \Vert_2),
\end{align*}
where $\phi_{\max}$ is the largest singular value of $\Sigma_0$.
\end{lemma}
The statistical performance of \textit{any} stationary point of the objective function (\ref{eqn:accelest}) is assured by the following theorem.
\begin{theorem}[Adapted from Theorem II.1. in \citet{elsener2018sharp}]
\label{thm:sharporacle}
Let $\tilde{\beta}$ be a stationary point. Let $\lambda_\varepsilon > 0$ such that for all $\beta' \in \mathcal{C}$ and a constant $0 \leq \gamma < 1$
\begin{align}
\label{eqn:empiricalcondition}
\left\vert \left( \dot{R}_n(\beta') - \dot{R}(\beta') \right)^T (\beta^0 - \beta') \right\vert \leq \lambda_\varepsilon \Vert \beta' - \beta^0 \Vert_1 + \gamma G(\tau(\beta' - \beta^0)).
\end{align}
Let $\lambda > \lambda_\varepsilon$ and $0 \leq \delta < 1$. Define $$\underline{\lambda} = \lambda - \lambda_\varepsilon, \quad \overline{\lambda} = \lambda + \lambda_\varepsilon + \delta \underline{\lambda}, \quad L = \frac{\overline{\lambda}}{(1-\delta) \underline{\lambda}}.$$
Then we have
\begin{align*}
\delta \underline{\lambda} \Vert \tilde{\beta} - \beta^0 \Vert_1 + R(\tilde{\beta}) \leq R(\beta^0) + \frac{C^2 \overline{\lambda}^2 s}{\Lambda_{\min}(\Sigma_0) (1-\gamma)}.
\end{align*}
\end{theorem}
We notice that Theorem \ref{thm:sharporacle} is a purely deterministic result on the statistical performance of any stationary point $\tilde{\beta}$ on $\mathcal{C}$. The tuning parameter $\lambda > 0$ needs to be chosen appropriately depending on the properties of the specific application.
\begin{remark}
We make use of the same terminology as in \citet{elsener2018sharp} for Condition \ref{eqn:empiricalcondition}: Empirical Process Condition. To verify the Empirical Process Condition we often need ad hoc techniques. They depend on the distributional assumptions on the observations and on the noise. 
\end{remark}

\section{Applications}
\label{s:applications}
In this section we present applications of the proposed method to some statistical estimation problems.
\subsection{Sparse PCA with missing entries}
\label{subsection:sparsepcaunifmissing}
Suppose we observe for $\left\lbrace \delta_{ij} \right\rbrace_{i = 1, \dots, n, j = 1, \dots p}$
\begin{align*}
Y_{ij} = \delta_{ij} X_{ij}.
\end{align*}
By using the usual sample covariance matrix one 
We correct the bias as in \citet{lounici2013sparse}, \citet{lounici2014high} and \citet{loh2011high} by defining
\begin{align}
\tilde{\Sigma} = \frac{1}{\delta^2}\frac{Y^TY}{n} - \frac{1 - \delta}{\delta^2} \text{diag} \left( \frac{Y^TY}{n} \right).
\end{align}
\begin{assumption}~
\label{assumption:pcamissingunif}
\begin{enumerate}[i)]
\item The random vectors $X_1, \dots, X_n \in \mathbb{R}^p$ are i.i.d. copies of a sub-Gaussian random vector $\tilde{X}$. Furthermore, it is assumed that there is a constant $c_1 > 0$ such that
\begin{align*}
\mathbb{E} \left[ (\tilde{X}u)^2 \right] \geq c_1 \left\Vert \tilde{X}u \right\Vert_{\psi_2}^2, \ \text{for all} \ u \in \mathbb{R}^p,
\end{align*}
where for $\psi_2 (x) = \exp(x^2) - 1$ the $\psi_2$ Orlicz norm of a random variable $\tilde{X}$ is defined as $\Vert \tilde{X} \Vert_{\psi_2} = \inf \lbrace C > 0: \mathbb{E} \psi_2(\vert \tilde{X} \vert/C) \leq 1 \rbrace$.
\item The random variables $\delta_{ij}$ are i.i.d. Bernoulli($\delta$) for $0 < \delta \leq 1$ and independent of $\tilde{X}$.
\end{enumerate}
\end{assumption}
The following lemma gives a high-probability bound on the random part (i.e. the noise) of the convex problem (\ref{eqn:initialest}).
\begin{lemma}[Proposition 3 in \citet{lounici2014high}]
Define
\begin{align*}
\mu(t) =&C \frac{\Vert \Sigma_0 \Vert_\infty}{c_1} \\
&\times \max \left\lbrace \sqrt{\frac{\mathbf{r}(\Sigma_0) (t + \log (2p))}{\delta^2 n}}, \frac{\mathbf{r}(\Sigma_0)(t + \log (2p))}{\delta^2 n} (c_1 \delta + t + \log n)  \right\rbrace,
\end{align*}
where $\mathbf{r}(\Sigma_0) = \text{\emph{trace}}(\Sigma_0)/\Vert \Sigma_0 \Vert_\infty$ is the effective rank of $\Sigma_0$.
Then we have for all $t > 1 \vee \log n$ with probability at least $1-\exp(-t)$ that
\begin{align*}
\left\Vert \tilde{\Sigma} - \Sigma_0 \right\Vert_\infty \leq \mu(t).
\end{align*}
\end{lemma}
\begin{remark}
We notice that the scaling with the effective number of observations in the previous lemma is $\delta^2 n$. This is sharper than the one derived e.g. in \citet{loh2011high} and in some of the subsequent sections of the present work.
\end{remark}
As a consequence of the previous lemma and Lemma \ref{lemma:initalest} we have the following corollary.
\begin{corollary}
We have that
\begin{align*}
\left\Vert \hat{\beta}_{\text{\emph{init}}} - \beta^0 \right\Vert_2 \leq \frac{1}{2 \sqrt{\Lambda_{\max}(\Sigma_0) - \alpha}} \alpha + 2 \sqrt{\varepsilon} \Lambda_{\max}(\Sigma_0)^{1/2}.
\end{align*}
with probability at least $ 1 - \exp(-\log(2p))$.
\end{corollary}
The statistical performance of any stationary point of the optimization problem (\ref{eqn:accelest}) depends on the bound on the ``noise term''. This bound is provided by the following lemma.
\begin{lemma}
\label{lemma:sparsepcamissing2}
Define  
\begin{align*}
\tilde{\lambda}_\varepsilon(t) = 4 C^2 \Vert \beta^0 \Vert_2 \left(2 \sqrt{\frac{2 (t + \log p)}{\delta^4 n}} + \frac{t + \log p}{\delta^2 n} \right).
\end{align*}
Let $\zeta > 0$ be a constant. Define $\Lambda_1 := 12 C^2 \Lambda_{\max}(\Sigma_Y)/(\phi_{\max} (\xi - 3 \eta))$ and
\begin{align*}
\gamma = \Lambda_1 \left( \left( \frac{12 \log (2p) + 16}{n} \right) \zeta^{-1} (1 + \zeta) + \zeta \right).
\end{align*}
We then have for all $\tilde{\beta} \in \mathcal{C}$
\begin{align*}
\left\vert \left( \dot{R}_n(\tilde{\beta}) - \dot{R}(\tilde{\beta}) \right)^T (\tilde{\beta} - \beta^0) \right\vert \leq& \tilde{\lambda}_\varepsilon(\log (2p)) \\
&+ 24C^2 Q \frac{16(\log(2p) + 4)}{2n \zeta} \Vert \tilde{\beta} - \beta^0 \Vert_1 \\
&+ 24 C^2 Q \frac{2(\log(2p) + 4)}{n} \Vert \tilde{\beta} - \beta^0 \Vert_1 \\
&+ \gamma G(\tau(\tilde{\beta} - \beta^0)).
\end{align*}
with probability at least $1 - 2 \exp(-\log(2p))$. Choosing 
\begin{equation*}
\zeta < (2 \Lambda_1)^{-1}
\end{equation*}
and assuming
\begin{equation*}
n > 2 \Lambda_1 (12 \log(2p) + 16) \zeta^{-1} (1 + \zeta)
\end{equation*}
we have that $\gamma < 1$. Therefore, the Empirical Process Condition (\ref{eqn:empiricalcondition}) is satisfied.
\end{lemma}
Combining Lemma \ref{lemma:sparsepcamissing2} with Theorem \ref{thm:sharporacle} we obtain the following corollary.
\begin{corollary}
Let $\tilde{\beta}$ be a stationary point of the objective function (\ref{eqn:accelest}). Let the assumptions in Lemma \ref{lemma:sparsepcamissing2} be satisfied. Define
\begin{align*}
\lambda_\varepsilon = \tilde{\lambda}_\varepsilon (\log (2p)) + 24 C^2 Q \left( \frac{16(\log(2p) + 4)}{2 n \zeta} + \frac{2(\log(2p) + 4)}{n} \right).
\end{align*}
Then we have with probability at least $1 - 2 \exp(- \log (2p))$ that
\begin{align*}
\delta \underline{\lambda} \Vert \tilde{\beta} - \beta^0 \Vert_1 + R(\tilde{\beta}) \leq R(\beta^0) + \frac{\overline{\lambda}^2 s_0}{8 \phi_{\max}(\rho - 3 \eta) (1 - \gamma)}.
\end{align*}
\end{corollary}

\subsection{Sparse PCA with multiplicative noise}
\label{subsection:sparsepcamultnoise}
This case is particularly important when the aim is to protect the customers' privacy. In \citet{hwang1986multiplicative} data collected by the U.S. Department of Energy on the energy consumption in the U.S. were considered. The data were artificially corrupted to avoid a possible identification of the citizens who participated in the survey. Other relevant areas are discussed in \citet{iturria1999polynomial}. There, 
Here we assume that the observed variables are given by
\begin{align*}
Y =  X \odot U,
\end{align*}
where  $\odot$ stays for an element-wise multiplication. The matrix $U$ consists of i.i.d. entries independent of $X$ such that the matrix $\mathbb{E} \left[U_1 U_1^T\right] =: M$ has only positive entries.

We correct the bias by defining
\begin{align}
\tilde{\Sigma} = \frac{Y^TY}{n} \div \mathbb{E} \left[ U_1 U_1^T \right] = \frac{Y^TY}{n} \div M
\end{align}
as in \citet{loh2011high} where $\div$ is to be read as an element-wise division.

\begin{assumption}~
\begin{enumerate}[i)]
\item The random vectors $X_1, \dots, X_n \in \mathbb{R}^p$ are i.i.d. copies of a sub-Gaussian random vector $\tilde{X}$.
\item The rows $U_i$ are i.i.d. with non-negative entries and expectation $\mathbb{E} \left[U_1 \right]$. The matrix $\mathbb{E} \left[U_1 U_1^T \right]$ is assumed to have positive entries. Furthermore, $U$ is assumed to be independent of $\tilde{X}$.
\end{enumerate}
\end{assumption}

\begin{lemma} 
\label{lemma:multnoiseinitial}
Denote $\Sigma_Y = \mathbb{E} Y^TY/n$. Define $\underset{\Vert \beta \Vert_2 \leq 1 }{\sup} \ \Vert \tilde{X} \beta \Vert_{\psi_2} =: C < \infty$ and for $ t > 0$
\begin{align*}
\mu(t) = &\frac{12 C^2}{\Vert M \Vert_{\min}} \sqrt{\frac{8(t + 2(\log(2p) + 4))}{n}} \Vert \Sigma_Y \Vert_\infty \\
&+ \frac{12 C^2}{\Vert M \Vert_{\min}} \sqrt{\frac{16 (\log(2p) + 4) )}{n}} \sqrt{\Vert \Sigma_Y \Vert_\infty} \\
&+ \frac{12 C^2}{\Vert M \Vert_{\min}} \left( \frac{t + 2(\log(2p) + 4)}{n} \right) \Vert \Sigma_Y \Vert_\infty \\
&+ \frac{12 C^2}{\Vert M \Vert_{\min}} \left( \frac{2(\log(2p) + 4}{n} \right).
\end{align*}
For $ t = \log(2p)$ we have that
\begin{align*}
\Vert \tilde{\Sigma} - \Sigma_0 \Vert_\infty \leq \mu(\log(2p))
\end{align*}
with probability at least $ 1 - \exp(-\log(2p))$.
\end{lemma}
Combining Lemma \ref{lemma:multnoiseinitial} with Lemma \ref{lemma:initalest} we obtain the following corollary.
\begin{corollary}
We have that
\begin{align*}
\left\Vert \hat{\beta}_{\text{\emph{init}}} - \beta^0 \right\Vert_2 \leq \frac{1}{2 \sqrt{\Lambda_{\max}(\Sigma_0) - \alpha}} \alpha + 2 \sqrt{\varepsilon} \Lambda_{\max}(\Sigma_0)^{1/2}.
\end{align*}
with probability at least $ 1 - \exp(-\log(2p))$.
\end{corollary}

\begin{remark}
The scaling with the inverse magnitude of the noise $1/\Vert M \Vert_{\min}$ is ``hidden'' in the quantity $\alpha$ which is defined in Lemma \ref{lemma:initalest}.
\end{remark}
We now proceed to the statistical properties of the nonconvex estimator. The next lemma establishes the main ingredient needed to apply Theorem \ref{thm:sharporacle}: the Empirical Process Condition \ref{eqn:empiricalcondition}.

\begin{lemma}
\label{lemma:multnoiseacc}
Define  
\begin{align*}
\tilde{\lambda}_\varepsilon(t) = 4 C^2 \Vert \beta^0 \Vert_2  \left(2 \sqrt{\frac{2 (t + \log p)}{\Vert M \Vert_{\min}^2 n}} + \frac{t + \log p}{\Vert M \Vert_{\min} n} \right).
\end{align*}
Let $\zeta > 0$ be a constant. Define $\Lambda_1 := 12 C^2 \Lambda_{\max}(\Sigma_Y)/(\phi_{\max} (\xi - 3 \eta)$ and
\begin{align*}
\gamma = \Lambda_1 \left( \left( \frac{12 \log (2p) + 16}{n} \right) \zeta^{-1} (1 + \zeta) + \zeta \right).
\end{align*}
We then have for all $\tilde{\beta} \in \mathcal{C}$
\begin{align*}
\left\vert \left( \dot{R}_n(\tilde{\beta}) - \dot{R}(\tilde{\beta}) \right)^T (\tilde{\beta} - \beta^0) \right\vert \leq& \tilde{\lambda}_\varepsilon(\log (2p)) \\
&+ 24C^2 Q \frac{16(\log(2p) + 4)}{2n \zeta} \Vert \tilde{\beta} - \beta^0 \Vert_1 \\
&+ 24 C^2 Q \frac{2(\log(2p) + 4)}{n} \Vert \tilde{\beta} - \beta^0 \Vert_1 \\
&+ \gamma G(\tau(\tilde{\beta} - \beta^0)).
\end{align*}
with probability at least $1 - 2 \exp(-\log(2p))$. Choosing 
\begin{equation*}
\zeta < (2 \Lambda_1)^{-1}
\end{equation*}
and assuming
\begin{equation*}
n > 2 \Lambda_1 (12 \log(2p) + 16) \zeta^{-1} (1 + \zeta)
\end{equation*}
we have that $\gamma < 1$. Therefore, the Empirical Process Condition (\ref{eqn:empiricalcondition}) is satisfied.
\end{lemma}
As a corollary we obtain the statistical performance of \textit{any} stationary point. We combine the previous lemma with Theorem \ref{thm:sharporacle}.
\begin{corollary}
Let $\tilde{\beta}$ be a stationary point of the objective function (\ref{eqn:accelest}). Let the assumptions in Lemma \ref{lemma:sparsepcamissing2} be satisfied. Define
\begin{align*}
\lambda_\varepsilon = \tilde{\lambda}_\varepsilon (\log (2p)) + 24 C^2 Q \left( \frac{16(\log(2p) + 4)}{2 n \zeta} + \frac{2(\log(2p) + 4)}{n} \right).
\end{align*}
Then we have with probability at least $1 - 2 \exp(- \log (2p))$ that
\begin{align*}
\delta \underline{\lambda} \Vert \tilde{\beta} - \beta^0 \Vert_1 + R(\tilde{\beta}) \leq R(\beta^0) + \frac{\overline{\lambda}^2 s_0}{8 \phi_{\max}(\rho - 3 \eta) (1 - \gamma)}.
\end{align*}
\end{corollary}

\subsection{Sparse PCA with non-uniform ``missingness''}
\label{subsection:nonuniformsparsepca}
The results in Subsection \ref{subsection:sparsepcamultnoise} can be applied to the case of ``non-uniformly'' missing data. This means that one can allow for different proportions of missing data in the rows.
As a special case of the previous formulation we might assume that
\begin{align*}
M_{ij}  = \left\lbrace \begin{array}{ll}
\delta_i \delta_j,  & \text{if} \ i \neq j, \\
\delta_i, & \text{if} \ i = j.
\end{array} \right.
\end{align*}

\begin{lemma}
\label{lemma:nonunifmissing}
Denote $\Sigma_Y = \mathbb{E} Y^TY/n$. Define $\underset{\Vert \beta \Vert_2 \leq 1 }{\sup} \ \Vert \tilde{X} \beta \Vert_{\psi_2} =: C < \infty$ and for $ t > 0$
\begin{align*}
\mu(t) = &\frac{12 C^2}{\delta_{\min}^2} \sqrt{\frac{8(t + 2(\log(2p) + 4))}{n}} \Vert \Sigma_Y \Vert_\infty \\
&+ \frac{12 C^2}{\delta_{\min}^2} \sqrt{\frac{16 (\log(2p) + 4) )}{n}} \sqrt{\Vert \Sigma_Y \Vert_\infty } \\
&+ \frac{12 C^2}{\delta_{\min}^2} \left( \frac{t + 2(\log(2p) + 4)}{n} \right) \Vert \Sigma_Y \Vert_\infty \\
&+ \frac{12 C^2}{\delta_{\min}^2} \left( \frac{2(\log(2p) + 4}{n} \right).
\end{align*}
For $ t = \log(2p)$ we have that
\begin{align*}
\Vert \tilde{\Sigma} - \Sigma_0 \Vert_\infty \leq \mu(\log(2p))
\end{align*}
with probability at least $ 1 - \exp(-\log(2p))$.
\end{lemma}
Lemma \ref{lemma:nonunifmissing} in combination with Lemma \ref{lemma:initalest} give the following corollary. 
\begin{corollary}
We have that
\begin{align*}
\left\Vert \hat{\beta}_{\text{\emph{init}}} - \beta^0 \right\Vert_2 \leq \frac{1}{2 \sqrt{\Lambda_{\max}(\Sigma_0) - \alpha}} \alpha + 2 \sqrt{\varepsilon} \Lambda_{\max}(\Sigma_0)^{1/2}.
\end{align*}
with probability at least $ 1 - \exp(-\log(2p))$.
\end{corollary}
The next lemma provides a bound on the random part of the problem.
\begin{lemma}
\label{lemma:nonunifmisspca}
Define  
\begin{align*}
\tilde{\lambda}_\varepsilon(t) = 4 C^2 \Vert \beta^0 \Vert_2  \left(2 \sqrt{\frac{2 (t + \log p)}{\delta_{\min}^4 n}} + \frac{t + \log p}{\delta_{\min}^2 n} \right).
\end{align*}
Let $\zeta > 0$ be a constant. Define $\Lambda_1 := 12 C^2 \Lambda_{\max}(\Sigma_Y)/(\phi_{\max} (\xi - 3 \eta)$ and
\begin{align*}
\gamma = \Lambda_1 \left( \left( \frac{12 \log (2p) + 16}{n} \right) \zeta^{-1} (1 + \zeta) + \zeta \right).
\end{align*}
We then have for all $\tilde{\beta} \in \mathcal{C}$
\begin{align*}
\left\vert \left( \dot{R}_n(\tilde{\beta}) - \dot{R}(\tilde{\beta}) \right)^T (\tilde{\beta} - \beta^0) \right\vert \leq& \tilde{\lambda}_\varepsilon(\log (2p)) \\
&+ 24C^2 Q \frac{16(\log(2p) + 4)}{2n \zeta} \Vert \tilde{\beta} - \beta^0 \Vert_1 \\
&+ 24 C^2 Q \frac{2(\log(2p) + 4)}{n} \Vert \tilde{\beta} - \beta^0 \Vert_1 \\
&+ \gamma G(\tau(\tilde{\beta} - \beta^0)).
\end{align*}
with probability at least $1 - 2 \exp(-\log(2p))$. Choosing 
\begin{equation*}
\zeta < (2 \Lambda_1)^{-1}
\end{equation*}
and assuming
\begin{equation*}
n > 2 \Lambda_1 (12 \log(2p) + 16) \zeta^{-1} (1 + \zeta)
\end{equation*}
we have that $\gamma < 1$. Therefore, the Empirical Process Condition (\ref{eqn:empiricalcondition}) is satisfied.
\end{lemma}
Unifying the previous lemma with Theorem \ref{thm:sharporacle} we have the following corollary.
\begin{corollary}
Let $\tilde{\beta}$ be a stationary point of the objective function (\ref{eqn:accelest}). Let the assumptions in Lemma \ref{lemma:sparsepcamissing2} be satisfied. Define
\begin{align*}
\lambda_\varepsilon = \tilde{\lambda}_\varepsilon (\log (2p)) + 24 C^2 Q \left( \frac{16(\log(2p) + 4)}{2 n \zeta} + \frac{2(\log(2p) + 4)}{n} \right).
\end{align*}
Then we have with probability at least $1 - 2 \exp(- \log (2p))$ that
\begin{align*}
\delta \underline{\lambda} \Vert \tilde{\beta} - \beta^0 \Vert_1 + R(\tilde{\beta}) \leq R(\beta^0) + \frac{\overline{\lambda}^2 s_0}{8 \phi_{\max}(\rho - 3 \eta) (1 - \gamma)}.
\end{align*}
\end{corollary}
\subsection{Low-rank matrices with missing data}
\label{subsection:lowrankmissingdata}
We start with the case of full observations. Assume a model of the form
\begin{align}
Y = X + W,
\end{align}
where $X$ is assumed to have a low-rank and $W$ is a matrix with i.i.d. rows with known covariance $\Sigma_W$ that perturbs the observed entries. In this work we are not interested in estimating/predicting the missing entries of $X$ based on the available noisy observation. We are rather interested in estimating the right singular vector of $X$ that corresponds to the largest eigenvalue of $X^TX$. Indeed, the singular value decomposition of $X$ is 
\begin{align*}
X = U D V^T.
\end{align*}
We are interested in estimating the (sparse) vector $V_1$. This is equivalent to computing the eigenvalue decomposition of $X^TX$:
\begin{align*}
X^TX = V D^2 V^T.
\end{align*}
We use the following unbiased estimator for $X^TX$:
\begin{align*}
\tilde{\Sigma} = Y^TY- \Sigma_W,
\end{align*}
where, for the sake of clarity, we assume the matrix $\Sigma_W$ to be known. We emphasize, however, that one might use approaches such as the ones discussed in \citet{loh2011high} to handle also the (more natural) case where $\Sigma_W$ is unknown.

We now move to the missing-data case. We observe
\begin{align*}
\tilde{Y}_{ij} = \delta_{ij} Y_{ij},
\end{align*}
where $\delta_{ij}$ are i.i.d. Bernoulli($\delta$) random variables independent of $Y$.
Then, an unbiased estimator for $\Sigma_0:=X^TX$ is given by correcting for both the missing data and the additive noise:
\begin{align*}
\tilde{\Sigma} = \frac{1}{\delta^2}  \tilde{Y}^T\tilde{Y}  - \frac{1 - \delta}{\delta^2} \text{diag} \left( \tilde{Y}^T\tilde{Y} \right) - \Sigma_W.
\end{align*}
\begin{assumption}~
\begin{enumerate}[i)]
\item The matrix $X$ is fixed but unknown and is possibly low-rank.
\item The noise matrix $W$ consists of $n$ i.i.d. rows with covariance matrix $\Sigma_W$.
\item The masks $\delta_{ij}$ are i.i.d. Bernoulli($\delta$) and independent of $W$.
\end{enumerate}
\end{assumption}

\begin{lemma}
\label{lemma:lowrankmissingdata}
Denote $\Sigma_{\tilde{Y}} = \mathbb{E} \tilde{Y}^T \tilde{Y}$. Define $\underset{\Vert \beta \Vert_2 \leq 1 }{\sup} \ \Vert \tilde{X} \beta \Vert_{\psi_2} =: C < \infty$ and for $ t > 0$
\begin{align*}
\mu(t) =& 12 C^2 \sqrt{\frac{8(t + 2(\log(2p) + 4))}{ \delta^4 n}} \Vert \Sigma_Y \Vert_\infty \\
&+ 12 C^2 \sqrt{\frac{16(\log(2p) + 4)}{\delta^4 n}} \sqrt{\Vert \Sigma_Y \Vert_\infty } \\
&+ 12 C^2 \left( \frac{t + 2(\log(2p) + 4)}{\delta^2 n} \right) \Vert \Sigma_Y \Vert_\infty \\
&+ 12 C^2 \left( \frac{2(\log(2p) + 4)}{\delta^2 n} \right)
\end{align*}
For $ t = \log(2p)$ we have that
\begin{align*}
\Vert \tilde{\Sigma} - \Sigma_0 \Vert_\infty \leq \mu(\log(2p))
\end{align*}
with probability at least $ 1 - \exp(-\log(2p))$.
\end{lemma}

Lemma \ref{lemma:lowrankmissingdata} in combination with Lemma \ref{lemma:initalest} give the following corollary. 
\begin{corollary}
We have that
\begin{align*}
\left\Vert \hat{\beta}_{\text{\emph{init}}} - \beta^0 \right\Vert_2 \leq \frac{1}{2 \sqrt{\Lambda_{\max}(\Sigma_0) - \alpha}} \alpha + 2 \sqrt{\varepsilon} \Lambda_{\max}(\Sigma_0)^{1/2}.
\end{align*}
with probability at least $ 1 - \exp(-\log(2p))$.
\end{corollary}
The following lemma guarantees that the stochastic part of the nonconvex estimation problem is bounded with high probability.
\begin{lemma}
\label{lemma:lowrankmissingdatanonconvex}
Define  
\begin{align*}
\tilde{\lambda}_\varepsilon(t) = 4 C^2 \Vert \beta^0 \Vert_2  \left(2 \sqrt{\frac{2 (t + \log p)}{\delta^4 n}} + \frac{t + \log p}{\delta^2 n} \right).
\end{align*}
Let $\zeta > 0$ be a constant. Define $\Lambda_1 := 12 C^2 \Lambda_{\max}(\Sigma_{\tilde{Y}})/(\phi_{\max} (\xi - 3 \eta)$ and
\begin{align*}
\gamma = \Lambda_1 \left( \left( \frac{12 \log (2p) + 16}{n} \right) \zeta^{-1} (1 + \zeta) + \zeta \right).
\end{align*}
We then have for all $\tilde{\beta} \in \mathcal{C}$
\begin{align*}
\left\vert \left( \dot{R}_n(\tilde{\beta}) - \dot{R}(\tilde{\beta}) \right)^T (\tilde{\beta} - \beta^0) \right\vert \leq& \tilde{\lambda}_\varepsilon(\log (2p)) \\
&+ 24C^2 Q \frac{16(\log(2p) + 4)}{2n \zeta} \Vert \tilde{\beta} - \beta^0 \Vert_1 \\
&+ 24 C^2 Q \frac{2(\log(2p) + 4)}{n} \Vert \tilde{\beta} - \beta^0 \Vert_1 \\
&+ \gamma G(\tau(\tilde{\beta} - \beta^0)).
\end{align*}
with probability at least $1 - 2 \exp(-\log(2p))$. Choosing 
\begin{equation*}
\zeta < (2 \Lambda_1)^{-1}
\end{equation*}
and assuming
\begin{equation*}
n > 2 \Lambda_1 (12 \log(2p) + 16) \zeta^{-1} (1 + \zeta)
\end{equation*}
we have that $\gamma < 1$. Therefore, the Empirical Process Condition (\ref{eqn:empiricalcondition}) is satisfied.
\end{lemma}
Combining Lemma \ref{lemma:lowrankmissingdatanonconvex} and Theorem \ref{thm:sharporacle} we arrive at the following corollary.
\begin{corollary}
Let $\tilde{\beta}$ be a stationary point of the objective function (\ref{eqn:accelest}). Let the assumptions in Lemma \ref{lemma:sparsepcamissing2} be satisfied. Define
\begin{align*}
\lambda_\varepsilon = \tilde{\lambda}_\varepsilon (\log (2p)) + 24 C^2 Q \left( \frac{16(\log(2p) + 4)}{2 n \zeta} + \frac{2(\log(2p) + 4)}{n} \right).
\end{align*}
Then we have with probability at least $1 - 2 \exp(- \log (2p))$ that
\begin{align*}
\delta \underline{\lambda} \Vert \tilde{\beta} - \beta^0 \Vert_1 + R(\tilde{\beta}) \leq R(\beta^0) + \frac{\overline{\lambda}^2 s_0}{8 \phi_{\max}(\rho - 3 \eta) (1 - \gamma)}.
\end{align*}
\end{corollary}

\section{Empirical results}
\label{s:empirical}
We simulate for $i = 1, \dots, n$ the i.i.d. vectors $X_i \sim \mathcal{N}(0,\Sigma_0)$ with
\begin{align*}
\Sigma_0 = \omega \beta^0 \beta^{0^T} + I_{p \times p},
\end{align*}
where $\omega \in \left\lbrace 0.2, 0.4, 0.6, 0.8, \dots, 2 \right\rbrace$ and
\begin{align*}
\beta^0 = (1, 1, 1, 1, 0, \dots, 0).
\end{align*}
The sample size and dimension are taken to be $n = p = 200$. The random variables $\delta_{i,j}$ are taken to be i.i.d. Bernoulli$(\delta)$, where $\delta$ lies in
\begin{align*}
\delta \in \left\lbrace 0.55, 0.60, 0.65, 0.70, 0.75, 0.80, 0.85, 0.90, 0.95, 1.00 \right\rbrace,
\end{align*}
where $\delta = 0.55$ means that about $55\%$ of the entries are observed and $\delta = 1.00$ means that \textit{all} entries are observed. Every point corresponds to an average of $200$ simulations. The estimation error is given by
\begin{align*}
\text{Estimation error}(\hat{\beta}) = \left\Vert \frac{\hat{\beta} \hat{\beta}^T}{\Vert \hat{\beta} \Vert_2^2} - \frac{\beta_0 \beta_0^T}{\Vert \beta_0 \Vert_2^2} \right\Vert_F.
\end{align*}
In panel a) of Figure \ref{fig:simulation1} we can observe that in this setting PCA fails to consistently estimate the first principal component. This is not surprising as the estimates are linear combinations of \emph{all} variables in the model. Despite the fact that the setting is not (very) high-dimensional setting PCA does not perform very well. In panel b) of Figure \ref{fig:simulation1} the estimation performance of the initial estimator \ref{eqn:fantopeest} can be observed. As expected, the estimation problem becomes easier with a stronger signal and with more observations (or equivalently with fewer missing data).
\begin{figure}[h]
\centering
\begin{subfigure}[t]{0.45\textwidth}
\includegraphics[scale=0.09]{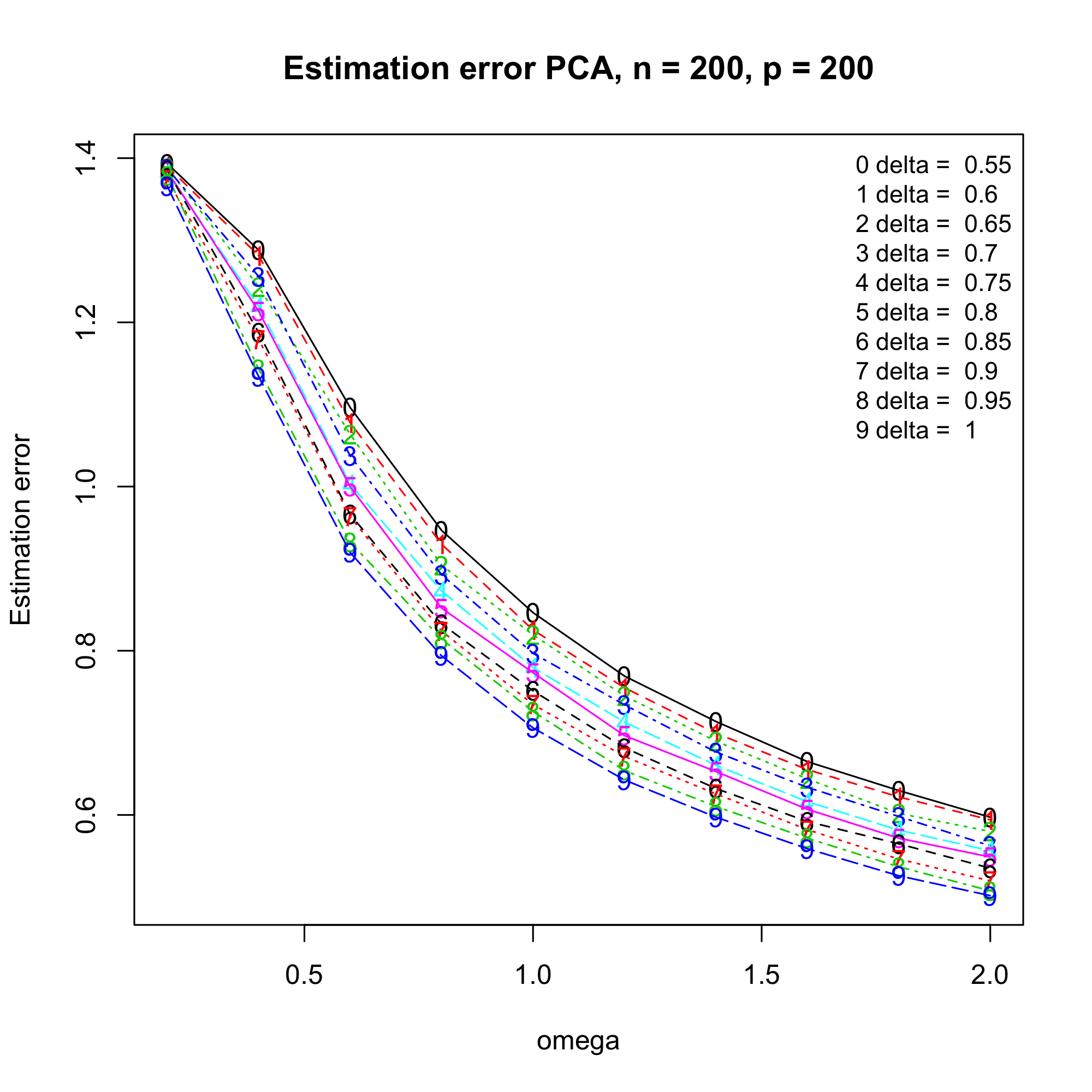}
\vspace{-0.5cm}
\caption{ }
\end{subfigure}
\begin{subfigure}[t]{0.45\textwidth}
\includegraphics[scale=0.09]{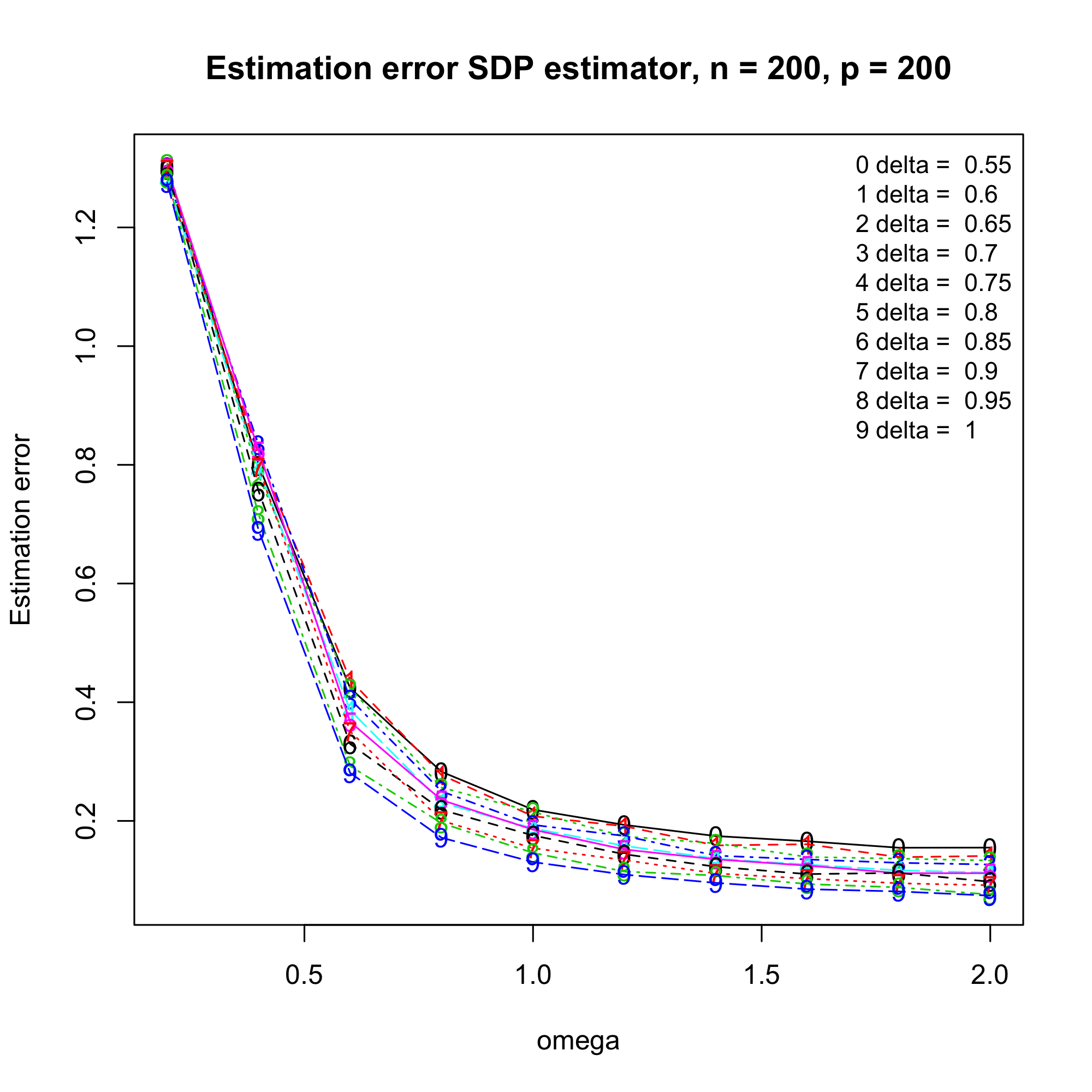}
\vspace{-0.5cm}
\caption{ }
\end{subfigure}
\caption{In panel a) the esimation performance of vanilla PCA and in panel b) the estimation performance of the initial estimator (\ref{eqn:fantopeest}) can be observed. We measure these performances depending on the signal strength $\omega$ and the proportion of observed data $\delta$.}
\label{fig:simulation1}
\end{figure}
We also tried to estimate the loadings of the first PC using the sample covariance estimator in the estimation procedure (\ref{eqn:fantopeest}), i.e. $\tilde{\Sigma} = X^T X/n$. As can be seen from Figure \ref{fig:sim} the estimator of the sample covariance matrix and therefore the ``output'' of the estimation procedure are affected by the missing data.
\begin{figure}
\centering
\includegraphics[scale=0.1]{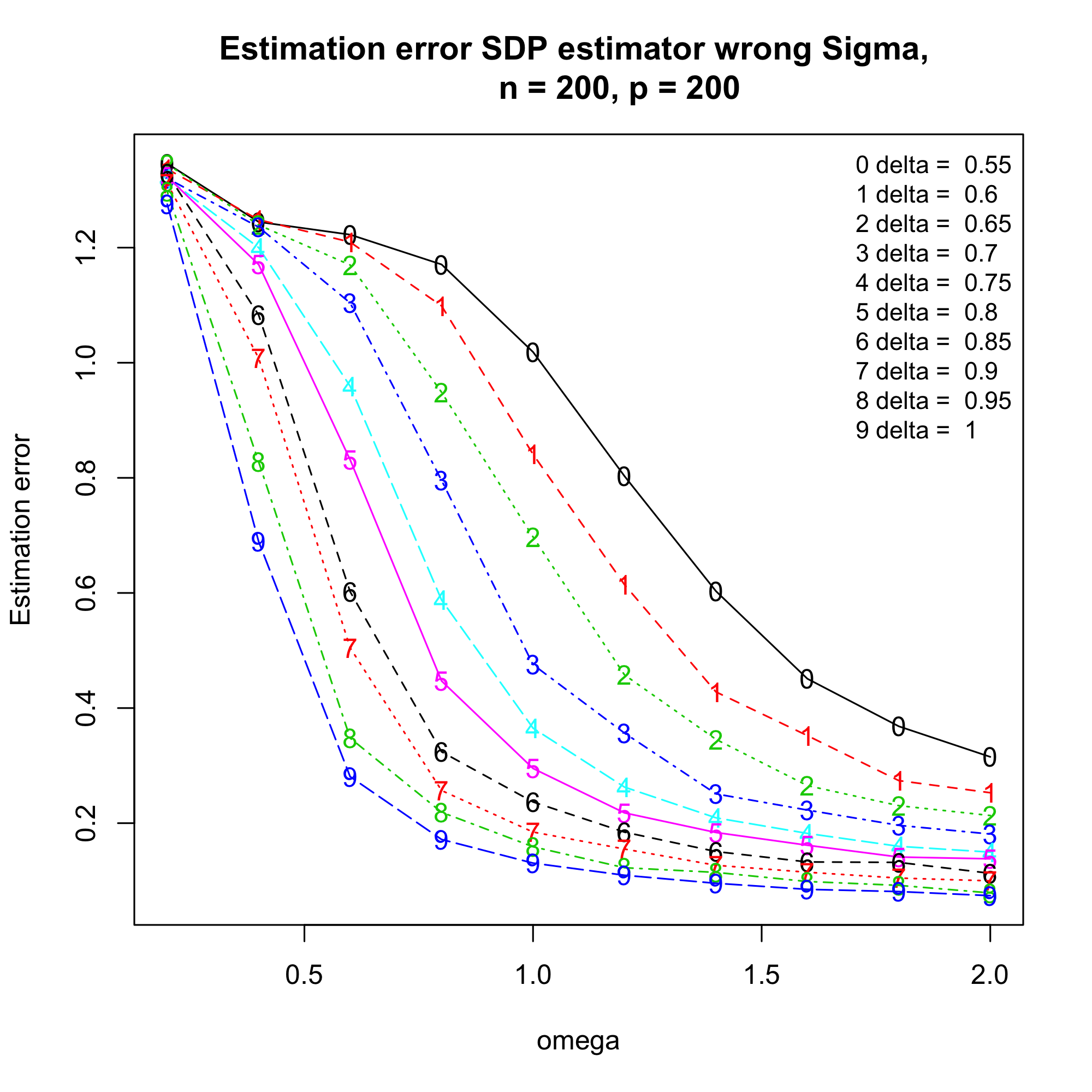}
\vspace{-0.5cm}
\caption{The estimation performance of the SDP estimator (\ref{eqn:fantopeest}) using $\tilde{\Sigma} = X^T X/n$. The performance is clearly worse with $X^TX/n$ than with the bias corrected versions.}
\label{fig:sim}
\end{figure}

Finally, Figure \ref{fig:comparisonhigh-dim} shows a direct comparison between the convex initial estimator, vanilla PCA and an imputation method from \citet{josse2016missmda}. It can be clearly observed that in the present setting the estimator (\ref{eqn:fantopeest}) outperforms vanilla PCA as well as the imputation method. We emphasize that it is a fair comparison as the dimensions of the problems are not ``too high'', namely $n = p = 200$.
\begin{figure}[h]
\begin{center}
\includegraphics[scale = 0.1]{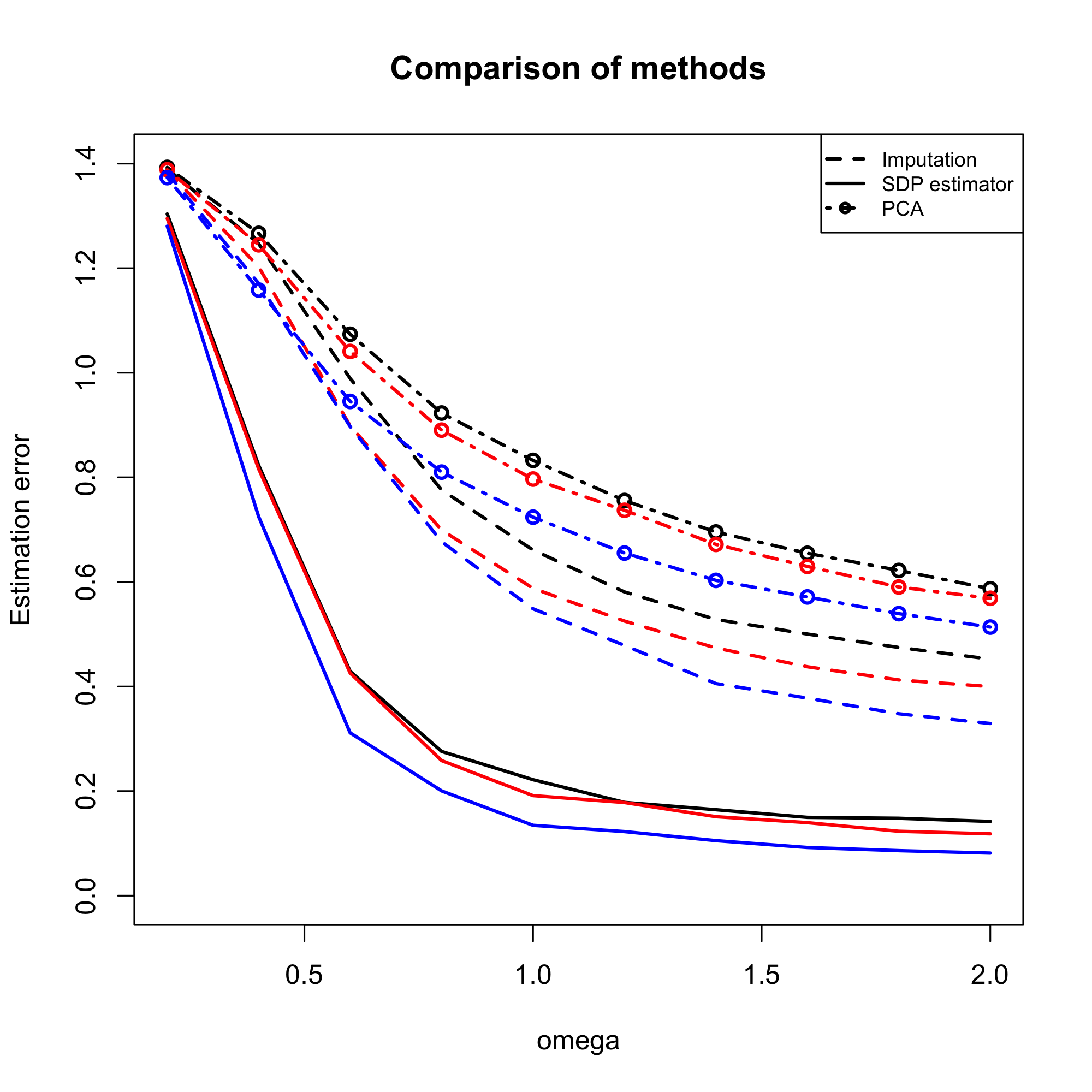}
\end{center}
\vspace{-0.5cm}
\caption{$n = p = 200$. Three different proportions of observed entries are considered: $\delta =  0.6$ (black lines), $\textcolor{red}{\delta = 0.7}$, $\textcolor{blue}{\delta = 0.95}$. The lines are averages of $200$ simulations.}
\label{fig:comparisonhigh-dim}
\end{figure}

\section{Discussion}
In this paper we have extended an existing convex relaxation of sparse Principal Component Analysis to the case of randomly missing data. We have demonstrated that this method can be applied also in different contexts than PCA. Due to the intrinsic computational limitations that come along with sparse PCA we have shown that a nonconvex acceleration might be used to improve the statistical performance. Further developments of the present framework might include sparse Canonical Correlation Analysis with missing data and introducing structured sparsity. Another future development of the present work might be similar as the one proposed in \citet{jankova2018biased}: with their methodological and theoretical results it might be possible to derive confidence regions under missing data. However, it has to be said that this could be computationally very intensive due to the estimation of the inverse Fisher information matrices needed to compute the debiased estimators.

\newpage 

\appendix
\section{Proofs}
\subsection{Proof of Lemma \ref{lemma:initalest}}
We parallel the proof of Lemma 2 in \citet{jankova2018biased}. At the end, the proof needs to be modified in order to match ``our'' initial estimator.

\begin{proof}[Proof of Lemma \ref{lemma:initalest}]
As far as the first part of the assertion is concerned, the proof works as in \citet{jankova2018biased}. The difference to our case lies in the derivation of upper bound for $\Vert \hat{\beta}_{\text{init}} - \beta^0 \Vert_2$. The eigendecomposition of the matrix $\hat{F}$ is given by
\begin{align*}
\hat{F} = \sum_{j = 1}^p \hat{\phi}_j^2 \hat{u}_j \hat{u}_j^T,
\end{align*}
where $\hat{\phi}_1 \geq \dots \geq \hat{\phi}_p$ and $\hat{u}_j^T\hat{u}_j = 1$.
We have that
\begin{align*}
&\left\vert \text{trace}(\tilde{\Sigma} \hat{F}) - \text{trace}(\Sigma_0 u_1 u_1^T) \right\vert &\\
&\leq \underbrace{\left\vert \text{trace} \left( \tilde{\Sigma} - \Sigma_0 \right) \hat{F} \right\vert}_{ = I} + \underbrace{\left\vert \text{trace} \left( \Sigma_0 \left( \sum_{j = 2}^p \hat{\phi}_j^2 \hat{u}_j \hat{u}_j^T \right) \right) \right\vert}_{= II} \\
&\phantom{aaa} + \underbrace{\left\vert \text{trace} \left(\hat{\phi}_1^2 \hat{u}_1 \hat{u}_1^T - u_1 u_1^T \right) \right\vert}_{= III}.
\end{align*}
As far as the first term is concerned, we have by the dual norm inequality
\begin{align*}
 I \leq \left\Vert \tilde{\Sigma} - \Sigma_0 \right\Vert_\infty \Vert \hat{F} \Vert_1 \leq \mu ( s + \varepsilon^2/\mu).
\end{align*}
As far as the second term is concerned, we have
\begin{align*}
II = 	\text{trace} \left( \Sigma_0 \left( \sum_{j = 2}^p \hat{\phi}_j^2 \hat{u}_j \hat{u}_j^T \right) \right) = \sum_{j = 2}^p \hat{\phi}_j^2 \hat{u}_j^T \Sigma_0 \hat{u}_j \\
\leq \Lambda_{\max}(\Sigma_0) \sum_{j = 2}^p \hat{\phi}_j^2 \leq \Lambda_{\max}(\Sigma_0) \varepsilon.
\end{align*}
As far as the third term is concerned, we have that
\begin{align*}
&III = \left\vert \text{trace} \left( \Sigma_0 (\hat{\phi}_1^2 \hat{u}_1 \hat{u}_1^T - u_1 u_1^T \right) \right\vert \\
&\leq \left\vert \text{trace} \left( \Sigma_0 (\hat{\phi}_1^2 -1) \hat{u}_1 \hat{u}_1^T \right) \right \vert  \\
&\phantom{aaaa} + \left\vert \text{trace} \left(\Sigma_0 (\hat{u}_1 \hat{u}_1^T - u_1 u_1^T \right) \right\vert \\
&\leq \varepsilon \Vert \hat{u}_1 \Vert_2 \Lambda_{\max}(\Sigma_0) + \left\vert (\hat{u}_1 - u_1)^T \Sigma_0 (\hat{u}_1 - u_1) \right\vert + \\
&\phantom{aaa} + 2 \left\vert u_1^T \Sigma_0 (\hat{u}_1 - u_1) \right\vert \\
&\leq \varepsilon \Lambda_{\max}(\Sigma_0) + \Lambda_{\max}(\Sigma_0) \Vert \hat{u}_1 - u_1 \Vert_2^2 \\
&\phantom{aaaa} + 2 \Lambda_{\max}(\Sigma_0)^{1/2} \Vert \hat{u}_1 - u_1 \Vert_2 \\
&\leq 5 \varepsilon \Lambda_{\max}(\Sigma_0) + 4 \sqrt{\varepsilon} \Lambda_{\max}(\Sigma_0)^{1/2}.
\end{align*}
We finally obtain
\begin{align*}
&\left\vert \text{trace} (\tilde{\Sigma} \hat{F}) - \text{trace}(\Sigma_0 u_1 u_1^T) \right\vert \\
&\leq \mu/2 (s + \varepsilon^2/2) + 2 \left(3 \Lambda_{\max}(\Sigma_0) \varepsilon + 2 \Lambda_{\max}(\Sigma_0)^{1/2} \sqrt{\varepsilon} \right) \\
&=: \alpha.
\end{align*}
By the triangle inequality this implies that
\begin{align*}
\left\vert \text{trace}(\Sigma_0 u_1 u_1^T)  \right\vert - \left\vert \text{trace}(\tilde{\Sigma} \hat{F})  \right\vert \leq \left\vert \left\vert \text{trace} (\tilde{\Sigma} \hat{F})  \right\vert - \left\vert \text{trace}(\Sigma_0 u_1 u_1^T) \right\vert  \right\vert \leq \alpha.
\end{align*} 
Multiplying both sides by $-1$ we obtain
\begin{align*}
 \left\vert \text{trace}(\tilde{\Sigma} \hat{F}) \right\vert - \left\vert \text{trace}(\Sigma_0 u_1 u_1^T)  \right\vert \geq - \alpha.
\end{align*}
Hence, 
\begin{align*}
\left\vert \text{trace}(\tilde{\Sigma} \hat{F}) \right\vert  \geq   \left\vert \text{trace}(\Sigma_0 u_1 u_1^T) \right\vert - \alpha.
\end{align*}
Then,
\begin{align*}
\left\vert \text{trace}(\tilde{\Sigma} \hat{F}) \right\vert^{1/2} \geq \sqrt{\left\vert \text{trace}(\Sigma_0 u_1 u_1^T) \right\vert - \alpha}.
\end{align*}
Therefore,
\begin{align*}
&\left\vert \left\vert \text{trace}(\tilde{\Sigma} \hat{F}) \right\vert^{1/2} + \left\vert \text{trace}(\Sigma_0 u_1 u_1^T) \right\vert^{1/2} \right\vert \\
&\geq \left\vert \sqrt{\left\vert \text{trace}(\Sigma_0 u_1 u_1^T) \right\vert - \alpha} + \left\vert \text{trace}(\Sigma_0 u_1 u_1^T) \right\vert^{1/2} \right\vert \\
&\geq 2 \sqrt{\left\vert \text{trace}(\Sigma_0 u_1 u_1^T) \right\vert - \alpha}.
\end{align*}
Then we have
\begin{align*}
\left\Vert \hat{\beta}_{\text{init}} - \beta^0 \right\Vert_2 &\leq \left\vert \left\vert \text{trace}(\tilde{\Sigma} \hat{F}) \right\vert^{1/2} - \left\vert \text{trace}(\Sigma_0 u_1 u_1^T) \right\vert^{1/2} \right\vert \Vert \hat{u}_1 \Vert_2 \\
&\phantom{aaa} + \left\vert \text{trace} (\Sigma_0 u_1 u_1^T) \right\vert^{1/2} \Vert \hat{u}_1 - u_1 \Vert_2 \\
&\leq \frac{1}{2 \sqrt{\Lambda_{\max}(\Sigma_0) - \alpha}} \alpha + 2 \sqrt{\varepsilon} \Lambda_{\max}(\Sigma_0)^{1/2}.
\end{align*}
\end{proof}

\subsection{Proofs of the lemmas in Subsection \ref{subsection:sparsepcaunifmissing}}
\begin{proof}[Proof of Lemma \ref{lemma:sparsepcamissing2}]
We notice that
\begin{align*}
&\left( \dot{R}_n(\tilde{\beta}) - \dot{R}(\tilde{\beta}) \right)^T (\tilde{\beta} - \beta^0) \\
&= \tilde{\beta}^T (\Sigma_0 - \tilde{\Sigma}) (\tilde{\beta} - \beta^0) \\
&= (\tilde{\beta} - \beta^0)^T(\Sigma_0 - \tilde{\Sigma}) (\tilde{\beta} - \beta^0) + \beta^{0^T}(\Sigma_0 - \tilde{\Sigma}) (\tilde{\beta} -\beta^0).
\end{align*}
We notice that
\begin{align*}
\tilde{\Sigma} - \Sigma_0 = \frac{Y^TY}{n} \div M - \Sigma_0 = \left( \frac{Y^TY}{n} - \Sigma_Y \right) \div M.
\end{align*}
Therefore, we have for all $v \in \mathbb{R}^p$ that
\begin{align*}
\left\vert v^T \left( \tilde{\Sigma} - \Sigma_0 \right) v \right\vert &= \left\vert v^T \left( \left( \frac{Y^TY}{n} - \Sigma_Y \right) \div M \right) v \right\vert \\
&\leq \frac{1}{\delta^2} \left\vert v^T \left( \frac{Y^TY}{n} - \Sigma_Y \right) v \right\vert
\end{align*}

To bound the previous quantity we invoke Lemma D.2 in \citet{elsener2018sharp}:
\begin{align*}
 &\left\vert (\tilde{\beta} - \beta^0)^T\left( \frac{Y^TY}{n} - \Sigma_Y \right) (\tilde{\beta} - \beta^0) \right\vert \\
 \leq& 12 C^2 \left( \frac{8(t + 2(\log(2p)+ 4))}{2 n \zeta} \right) (\tilde{\beta} - \beta^0)^T \Sigma_Y (\tilde{\beta} - \beta^0) \\
 &+ 12 C^2 \frac{\zeta}{2} (\tilde{\beta} - \beta^0)^T \Sigma_Y (\tilde{\beta} - \beta^0) \\
 &+ 12 C^2 \left( \frac{16 (\log(2p) + 4)}{2n \zeta} \Vert \tilde{\beta} - \beta^0 \Vert_1^2 \right) \\
 &+ 12 C^2 \frac{\zeta}{2} (\tilde{\beta} - \beta^0)^T \Sigma_Y (\tilde{\beta} - \beta^0) \\
 &+ 12 C^2 \left( \frac{t + 2(\log (2p) + 4)}{n} \right) (\tilde{\beta} - \beta^0)^T \Sigma_Y (\tilde{\beta} - \beta^0) \\
 &+ 12 C^2 \left( \frac{2(\log(2p) + 4)}{n} \right) \Vert \tilde{\beta} - \beta^0 \Vert_1^2. 
\end{align*}
With $t = \log(2p)$ and noticing that
\begin{align*}
&(\tilde{\beta} - \beta^0)^T \Sigma_Y (\tilde{\beta} - \beta^0) \\
&\leq \Lambda_{\max}(\Sigma_Y) \Vert \tilde{\beta} - \beta^0 \Vert_2^2 \\
& = \frac{\Lambda_{\max}(\Sigma_Y)}{2 \phi_{\max}(\rho - 3 \eta)} 2 \phi_{\max}(\rho - 3 \eta) \Vert \tilde{\beta} - \beta^0 \Vert_2^2 \\
&= \frac{\Lambda_{\max}(\Sigma_Y)}{2 \phi_{\max}(\rho - 3 \eta)} G(\Vert \tilde{\beta} - \beta^0 \Vert_2).
\end{align*}
we arrive at
\begin{align*}
\left\vert (\tilde{\beta} - \beta^0)^T \left( \frac{Y^TY}{n} - \Sigma_Y \right) (\tilde{\beta} - \beta^0)  \right\vert &\leq \gamma G(\Vert \tilde{\beta} - \beta^0 \Vert_2) \\
&+ 12 C^2 \frac{16(\log(2p) + 4)}{2n\zeta} \Vert \tilde{\beta} - \beta^0 \Vert_1^2 \\
&+ 12 C^2 \frac{2 (\log(2p) + 4)}{n} \Vert \tilde{\beta} - \beta^0 \Vert_1^2.
\end{align*}
Furthermore, we have that
\begin{align*}
\left\Vert (\tilde{\Sigma} - \Sigma_0) \beta^0 \right\Vert_\infty &= \left\Vert \left( \left( \frac{Y^TY}{n} - \Sigma_Y \right) \div M \right) \beta^0 \right\Vert_\infty \\
&\leq \frac{1}{\delta^2} \left\Vert \left( \frac{Y^TY}{n} - \Sigma_Y \right) \beta^0 \right\Vert_\infty.
\end{align*}
We also have that
\begin{align*}
\left\Vert \left( \frac{Y^TY}{n} - \Sigma_Y \right) \beta^0 \right\Vert_\infty = \underset{1 \leq j \leq p}{\max} \ \left\vert e_j^T \left( \frac{Y^TY}{n} - \Sigma_Y \right) \beta^0 \right\vert.
\end{align*}
For all $j = 1, \dots, p$ and all $t > 0$ we then have by Bernstein's inequality by noting that the rows of $Y$ are also sub-Gaussian with $\underset{\Vert \beta \Vert_2 \leq 1}{\sup} \ \Vert Y_i \beta \Vert_{\psi_2} = C$ for all $i = 1, \dots, n$. Hence,
\begin{align*}
&\left\vert e_j^T \left( \frac{Y^TY}{n} - \Sigma_Y \right) \beta^0 \right\vert \\
&\leq 8 C^2 \Vert \beta^0 \Vert_2 \sqrt{\frac{2 t}{n}} + 4 C^2 \Vert \beta^0 \Vert_2 \frac{t}{n}.
\end{align*}
By the union bound we then have
\begin{align*}
\left\Vert \left( \frac{Y^TY}{n} - \Sigma_Y \right) \beta^0 \right\Vert_\infty &\leq 8 C^2 \Vert \beta^0 \Vert_2 \sqrt{\frac{2(t + \log p)}{n}} + 4 C^2 \Vert \beta^0 \Vert_2 \frac{t + \log p}{n}.
\end{align*}

\end{proof}

\subsection{Proofs of the lemmas in Subsection \ref{subsection:sparsepcamultnoise}}

\begin{proof}[Proof of Lemma \ref{lemma:multnoiseinitial}]
As in the proof of Corollary 2 in the supplemental material of \citet{loh2011high} we notice that
\begin{align*}
\frac{Y^TY}{n} \div M - \Sigma_0 = \left( \frac{Y^TY}{n} - \Sigma_Y \right) \div M.
\end{align*}
We then have for all $v \in \mathbb{R}^p$
\begin{align*}
\left\vert v^T \left(\frac{Y^TY}{n} \div M - \Sigma_0 \right) v \right\vert &= \left\vert v^T \left( \left( \frac{Y^TY}{n} - \Sigma_Y \right) \div M \right) v \right\vert \\
&\leq \frac{1}{\Vert M \Vert_{\min} } \left\vert v^T \left( \frac{Y^TY}{n} - \Sigma_Y \right) v \right\vert \\
\end{align*}
To upper bound the latter term we make use of Lemma D.2. in \citet{elsener2018sharp}. By that lemma we have
\begin{align*}
\left\vert v^T \left(\frac{Y^TY}{n} - \Sigma_Y \right)v \right\vert &\leq 12 C^2 \sqrt{\frac{8(t + 2(\log(2p) + 4))}{n}} v^T \Sigma_Yv \\
&+ 12 C^2 \sqrt{\frac{16(\log(2p) + 4))}{n}} \Vert u \Vert_1 \sqrt{v^T \Sigma_Y v} \\
&+ 12 C^2 \left( \frac{t + 2(\log (2p) + 4)}{n} \right) v^T \Sigma_Y v \\
&+ 12 C^2 \left( \frac{2 (\log(2p) + 4)}{n} \right) \Vert v \Vert_1^2
\end{align*}
with probability at least $1- \exp(-t)$ for all $t > 0$.
We now choose $v = e_j$ and take the maximum over $j$ on both sides of the inequality:
\begin{align*}
\Vert \tilde{\Sigma} - \Sigma_0 \Vert_\infty &\leq \frac{12 C^2}{\Vert M \Vert_{\min}} \sqrt{\frac{8(t + 2(\log(2p) + 4))}{n}} \Vert \Sigma_Y \Vert_\infty \\
&+ \frac{12 C^2}{\Vert M \Vert_{\min}} \sqrt{\frac{16 (\log(2p) + 4) )}{n}} \sqrt{\Vert \Sigma_Y \Vert_\infty} \\
&+ \frac{12 C^2}{\Vert M \Vert_{\min}} \left( \frac{t + 2(\log(2p) + 4)}{n} \right) \Vert \Sigma_Y \Vert_\infty \\
&+ \frac{12 C^2}{\Vert M \Vert_{\min}} \left( \frac{2(\log(2p) + 4}{n} \right).
\end{align*}
\end{proof}

\begin{proof}[Proof of Lemma \ref{lemma:multnoiseacc}]
We notice that
\begin{align*}
&\left( \dot{R}_n(\tilde{\beta}) - \dot{R}(\tilde{\beta}) \right)^T (\tilde{\beta} - \beta^0) \\
&= \tilde{\beta}^T (\Sigma_0 - \tilde{\Sigma}) (\tilde{\beta} - \beta^0) \\
&= (\tilde{\beta} - \beta^0)^T(\Sigma_0 - \tilde{\Sigma}) (\tilde{\beta} - \beta^0) + \beta^{0^T}(\Sigma_0 - \tilde{\Sigma}) (\tilde{\beta} -\beta^0).
\end{align*}
As far as the first term is considered, we proceed as in the proof of Lemma \ref{lemma:multnoiseinitial}:
\begin{align*}
\left\vert v^T \left(\frac{Y^TY}{n} \div M - \Sigma_0 \right) v \right\vert &= \left\vert v^T \left( \left( \frac{Y^TY}{n} - \Sigma_Y \right) \div M \right) v \right\vert \\
&\leq \frac{1}{\Vert M \Vert_{\min} } \left\vert v^T \left( \frac{Y^TY}{n} - \Sigma_Y \right) v \right\vert \\
\end{align*}
To bound the previous quantity we invoke Lemma D.2 in \citet{elsener2018sharp}:
\begin{align*}
 &\left\vert (\tilde{\beta} - \beta^0)^T\left( \frac{Y^TY}{n} - \Sigma_Y \right) (\tilde{\beta} - \beta^0) \right\vert \\
 \leq& 12 C^2 \left( \frac{8(t + 2(\log(2p)+ 4))}{2 n \zeta} \right) (\tilde{\beta} - \beta^0)^T \Sigma_Y (\tilde{\beta} - \beta^0) \\
 &+ 12 C^2 \frac{\zeta}{2} (\tilde{\beta} - \beta^0)^T \Sigma_Y (\tilde{\beta} - \beta^0) \\
 &+ 12 C^2 \left( \frac{16 (\log(2p) + 4)}{2n \zeta} \Vert \tilde{\beta} - \beta^0 \Vert_1^2 \right) \\
 &+ 12 C^2 \frac{\zeta}{2} (\tilde{\beta} - \beta^0)^T \Sigma_Y (\tilde{\beta} - \beta^0) \\
 &+ 12 C^2 \left( \frac{t + 2(\log (2p) + 4)}{n} \right) (\tilde{\beta} - \beta^0)^T \Sigma_Y (\tilde{\beta} - \beta^0) \\
 &+ 12 C^2 \left( \frac{2(\log(2p) + 4)}{n} \right) \Vert \tilde{\beta} - \beta^0 \Vert_1^2. 
\end{align*}
With $t = \log(2p)$ and noticing that
\begin{align*}
&(\tilde{\beta} - \beta^0)^T \Sigma_Y (\tilde{\beta} - \beta^0) \\
&\leq \Lambda_{\max}(\Sigma_Y) \Vert \tilde{\beta} - \beta^0 \Vert_2^2 \\
& = \frac{\Lambda_{\max}(\Sigma_Y)}{2 \phi_{\max}(\rho - 3 \eta)} 2 \phi_{\max}(\rho - 3 \eta) \Vert \tilde{\beta} - \beta^0 \Vert_2^2 \\
&= \frac{\Lambda_{\max}(\Sigma_Y)}{2 \phi_{\max}(\rho - 3 \eta)} G(\Vert \tilde{\beta} - \beta^0 \Vert_2).
\end{align*}
we arrive at
\begin{align*}
\left\vert (\tilde{\beta} - \beta^0)^T \left( \frac{Y^TY}{n} - \Sigma_Y \right) (\tilde{\beta} - \beta^0)  \right\vert &\leq \gamma G(\Vert \tilde{\beta} - \beta^0 \Vert_2) \\
&+ 12 C^2 \frac{16(\log(2p) + 4)}{2n\zeta} \Vert \tilde{\beta} - \beta^0 \Vert_1^2 \\
&+ 12 C^2 \frac{2 (\log(2p) + 4)}{n} \Vert \tilde{\beta} - \beta^0 \Vert_1^2.
\end{align*}
Furthermore, we have that
\begin{align*}
\left\Vert (\tilde{\Sigma} - \Sigma_0) \beta^0 \right\Vert_\infty &= \left\Vert \left( \left( \frac{Y^TY}{n} - \Sigma_Y \right) \div M \right) \beta^0 \right\Vert_\infty \\
&\leq \frac{1}{\Vert M \Vert_{\min}} \left\Vert \left( \frac{Y^TY}{n} - \Sigma_Y \right) \beta^0 \right\Vert_\infty.
\end{align*}
We also have that
\begin{align*}
\left\Vert \left( \frac{Y^TY}{n} - \Sigma_Y \right) \beta^0 \right\Vert_\infty = \underset{1 \leq j \leq p}{\max} \ \left\vert e_j^T \left( \frac{Y^TY}{n} - \Sigma_Y \right) \beta^0 \right\vert.
\end{align*}
For all $j = 1, \dots, p$ and all $t > 0$ we then have by Bernstein's inequality by noting that the rows of $Y$ are also sub-Gaussian with $\underset{\Vert \beta \Vert_2 \leq 1}{\sup} \ \Vert Y_i \beta \Vert_{\psi_2} = C$ for all $i = 1, \dots, n$. Hence,
\begin{align*}
&\left\vert e_j^T \left( \frac{Y^TY}{n} - \Sigma_Y \right) \beta^0 \right\vert \\
&\leq 8 C^2 \Vert \beta^0 \Vert_2 \sqrt{\frac{2 t}{n}} + 4 C^2 \Vert \beta^0 \Vert_2 \frac{t}{n}.
\end{align*}
By the union bound we then have
\begin{align*}
\left\Vert \left( \frac{Y^TY}{n} - \Sigma_Y \right) \beta^0 \right\Vert_\infty &\leq 8 C^2 \Vert \beta^0 \Vert_2 \sqrt{\frac{2(t + \log p)}{n}} + 4 C^2 \Vert \beta^0 \Vert_2 \frac{t + \log p}{n}.
\end{align*}
\end{proof}

\subsection{Proofs of the lemmas in Subseciton \ref{subsection:nonuniformsparsepca}}
\begin{proof}[Proof of Lemma \ref{lemma:nonunifmissing}]
As in the proof of Corollary 2 in the supplemental material of \citet{loh2011high} we notice that
\begin{align*}
\frac{Y^TY}{n} \div M - \Sigma_0 = \left( \frac{Y^TY}{n} - \Sigma_Y \right) \div M
\end{align*}
We then have for all $v \in \mathbb{R}^p$
\begin{align*}
\left\vert v^T \left(\frac{Y^TY}{n} \div M - \Sigma_0 \right) v \right\vert &= \left\vert v^T \left( \left( \frac{Y^TY}{n} - \Sigma_Y \right) \div M \right) v \right\vert \\
&\leq \frac{1}{\Vert M \Vert_{\min} } \left\vert v^T \left( \frac{Y^TY}{n} - \Sigma_Y \right) v \right\vert \\
&\leq \frac{1}{\delta_{\min}^2} \left\vert v^T \left( \frac{Y^TY}{n} - \Sigma_Y \right) v \right\vert.
\end{align*}

\end{proof}
\begin{proof}[Proof of Lemma \ref{lemma:nonunifmisspca}]
The proof is identical to the proof of Lemma \ref{lemma:multnoiseacc} up to the fact that we have the following bound
\begin{align*}
\frac{1}{\Vert M \Vert_{\min}} \leq \frac{1}{\delta_{\min}^2}.
\end{align*}
\end{proof}
\subsection{Proofs of the lemmas in Subsection \ref{subsection:lowrankmissingdata}}

\begin{proof}[Proof of Lemma \ref{lemma:lowrankmissingdata}]
We have that
\begin{align*}
\Vert \tilde{\Sigma} - \Sigma_0 \Vert_\infty &= \left\Vert \tilde{Y}^T \tilde{Y} \div M - (\Sigma_0 + \Sigma_W) \right\Vert_\infty \\
&= \left\Vert \left(\tilde{Y}^T \tilde{Y} - \Sigma_{\tilde{Y}} \right) \div M \right\Vert_\infty \\
&\leq \frac{1}{\delta^2} \left\Vert \tilde{Y}^T \tilde{Y} - \Sigma_{\tilde{Y}} \right\Vert_\infty
\end{align*}
By Lemma D.2 in \citet{elsener2018sharp} we have for all $v \in \mathbb{R}^p$ that
\begin{align*}
\left\vert v^T \left( \tilde{Y}^T \tilde{Y} - \Sigma_{\tilde{Y}} \right) v \right\vert \leq& 12 C^2 \sqrt{\frac{8(t + 2(\log(2p) + 4))}{n}} v^T \Sigma_{\tilde{Y}} v \\
&+ 12 C^2 \sqrt{\frac{16(\log(2p) + 4))}{n}} \Vert v \Vert_1 \sqrt{v^T \Sigma_{\tilde{Y}} v} \\
&+ 12 C^2 \left( \frac{t + 2(\log(2p) + 4)}{n} \right) v^T \Sigma_{\tilde{Y}} v \\
&+ 12 C^2 \left( \frac{2(\log(2p) + 4)}{n} \right) \Vert v \Vert_1^2
\end{align*}
with probability at least $1 - \exp(-t)$ for all $t > 0$.
Hence, taking $v = e_j$ and taking the maximum over all $j$ on both sides we obtain
\begin{align*}
\Vert \tilde{\Sigma} - \Sigma_0 \Vert_\infty \leq& 12 C^2 \sqrt{\frac{8(t + 2(\log(2p) + 4))}{ \delta^4 n}} \Vert \Sigma_{\tilde{Y}} \Vert_\infty \\
&+ 12 C^2 \sqrt{\frac{16(\log(2p) + 4)}{\delta^4 n}} \sqrt{\Vert \Sigma_{\tilde{Y}} \Vert_\infty} \\
&+ 12 C^2 \left( \frac{t + 2(\log(2p) + 4)}{\delta^2 n} \right) \Vert \Sigma_{\tilde{Y}} \Vert_\infty \\
&+ 12 C^2 \left( \frac{2(\log(2p) + 4)}{\delta^2 n} \right)
\end{align*}
\end{proof}

\begin{proof}[Proof of Lemma \ref{lemma:lowrankmissingdatanonconvex}]
The proof is identical to the proof of Lemma \ref{lemma:multnoiseacc} up to the fact that we have the following bound
\begin{align*}
\frac{1}{\Vert M \Vert_{\min}} \leq \frac{1}{\delta^2}.
\end{align*}
\end{proof}
\newpage

\bibliographystyle{apalike}
\bibliography{/Users/anelsene/Dokumente/sparse-subspace-missing-data/myreferences.bib}

\end{document}